\documentclass[12pt]{amsart}
\usepackage{amsfonts, amssymb, mathrsfs, amsmath, amsthm}
\usepackage{fullpage}
\usepackage[colorlinks=true]{hyperref}

\usepackage{comment}
    \usepackage{url}
    \usepackage{arydshln}

\makeatletter
\let\ams@starttoc\@starttoc
\makeatother
\usepackage[parfill]{parskip}
\makeatletter
\let\@starttoc\ams@starttoc
\patchcmd{\@starttoc}{\makeatletter}{\makeatletter\parskip\z@}{}{}
\makeatother
\usepackage[colorlinks=true]{hyperref}
\usepackage{cleveref}
\usepackage{thm-restate}
\theoremstyle{definition}

\newtheorem{definition}{Definition}[section]
\newtheorem{theorem}[definition]{Theorem}

\newtheorem{corollary}[definition]{Corollary}
\newtheorem{proposition}[definition]{Proposition}
\newtheorem{remark}[definition]{Remark}
\newtheorem{example}[definition]{Example}
\newtheorem{conjecture}[definition]{Conjecture}

\newtheorem{question}[definition]{Question}
\newtheorem{formula}[definition]{Formula}
\crefname{section}{Section}{Sections}
\crefname{definition}{Definition}{Definitions}
\crefname{theorem}{Theorem}{Theorems}
\crefname{lemma}{Lemma}{Lemmas}
\crefname{corollary}{Corollary}{Corollaries}
\crefname{proposition}{Proposition}{Propositions}
\crefname{remark}{Remark}{Remarks}
\crefname{example}{Example}{Examples}
\crefname{conjecture}{Conjecture}{Conjectures}
\crefname{problem}{Problem}{Problems}
\crefname{question}{Question}{Questions}
\crefname{formula}{Formula}{Formulas}

\newcommand{\NN}{\mathbb{N}}
\newcommand{\ZZ}{\mathbb{Z}}

\newcommand{\KK}{\mathbb{K}}

\newcommand{\AAA}{\mathcal{A}}
\newcommand{\BBB}{\mathcal{B}}

\newcommand{\Der}{\operatorname{Der}}
\newcommand{\Ker}{\operatorname{Ker}}

\newcommand{\Sym}{\operatorname{Sym}}
\newcommand{\Hom}{\operatorname{Hom}}
\newcommand{\sign}{\operatorname{sign}}
\newcommand{\suchthat}{\text{s.t.}}

\newcommand{\vecxxxx}[4]{\begin{pmatrix}#1\\#2\\#3\\#4\end{pmatrix}}

\newcommand{\cevxx}[2]{\begin{pmatrix}#1&#2\end{pmatrix}}

\newcommand{\hprod}{\circ}

\newcommand{\tuple}[1]{({#1}:0)^{\text{tuple}}}
\newcommand{\tuplee}[2]{({#1}:{#2})^{\text{tuple}}}

\newcommand{\even}[1]{{#1}_\text{even}}
\newcommand{\odd}[1]{{#1}_\text{odd}}

\newcommand{\length}[1]{\operatorname{len}(#1)}

\usepackage{tcolorbox}
\newcommand{\red}[1]{\textcolor{red}{#1}}

\newcommand{\gray}[1]{\textcolor{gray}{#1}}

\title[Exponents of $2$-multiarrangements and Wakefield--Yuzvinsky matrices]{Exponents of $2$-multiarrangements and Wakefield--Yuzvinsky matrices}
\author[S.~Maehara]{Shota Maehara}
\email{\vspace{-5mm}
maehara.shota.027@s.kyushu-u.ac.jp}
\address{Joint Graduate School of Mathematics for Innovation, Kyushu University, 744 Motooka Nishi-ku Fukuoka 819-0395, Japan}

\begin{document}

\begin{abstract}
    In the theory of hyperplane arrangements, M. Wakefield and S. Yuzvinsky utilized a square matrix in their research on the exponents of $2$-dimensional multiarrangements. Using such a matrix, they showed that the exponents of $2$-dimensional multiarrangements are as close as possible in general position for any fixed balanced multiplicity. In this article, we introduce a matrix similar to that of Wakefield and Yuzvinsky and explore further applications to the exponents. 
    In fact, the exponents of $2$-dimensional multiarrangements are 
    determined by whether the corresponding matrices have full rank. 
    As one of our main results, 
    we introduce a new class of $2$-dimensional arrangements 
    for which the exponents are as close as possible for any balanced multiplicities, 
    except for the constant one multiplicity. 
    %
    We also proceed with the classification of $B_2$-exponents, 
    and we provide an alternative proof for some known results on the exponents. 
    \end{abstract}

\maketitle

\section{Introduction} 
Let $\KK$ be a field of characteristic zero, 
$V:=\KK^\ell$, $V^*:=\Hom_\KK(V,\KK)$ and $S:=\Sym(V^*)$. 
By fixing a $\KK$-linear basis $\{x_1,\ldots,x_\ell\}$ of $V^*$, 
we identify $S$ with a polynomial ring $\KK[x_1,\ldots,x_\ell]$. 
For the polynomial ring $S$, 
we define the module $\Der S$ 
of $\KK$-linear derivations on $S$ as a free $S$-module generated by 
$\left\{\partial_{x_i}:=\frac{\partial}{\partial{x_i}} \mid 1 \leq i \leq \ell\right\}$, 
where $\partial_{x_i}(x_j)$ is the Kronecker delta $\delta_{ij}$ 
and every $\psi\in\Der S$ satisfies 
$\psi(fg)=\psi(f)g+f\psi(g)$ for all polynomials $f,g\in S$. 
Defining a \emph{hyperplane} as a 
vector subspace in $V$ of codimension one, 
we call a finite (non-empty) 
collection of hyperplanes in $V=\KK^\ell$ an 
\emph{$\ell$-dimensional hyperplane arrangement}. 
%
Taking multiplicities into account, the concept of 
hyperplane arrangements has been generalized to that of multiarrangements by Ziegler \cite{Ziegler}. 
More precisely, 
a function $m$ from a hyperplane arrangement 
$\AAA$ to $\NN=\ZZ_{>0}$ is called a \emph{multiplicity} on this arrangement, 
and the \emph{multiarrangement} $(\AAA,m)$ is defined 
as the pair of $\AAA$ and $m$. 
When $\AAA$ is an $\ell$-dimensional arrangement, 
we sometimes call $(\AAA,m)$ an \emph{$\ell$-multiarrangement} for short. 
We often express the multiarrangement $(\AAA,m)$ by its \emph{defining polynomial} 
$Q(\AAA,m):=\prod_{H\in\AAA}\alpha_H^{m(H)}$, fixing each $\alpha_H\in V^*$ 
that satisfies $H=\Ker(\alpha_H)$. 
For a multiarrangement $(\AAA,m)$, 
let us define the \emph{size} of the multiplicity by $|m|:=\sum_{H\in\AAA}m(H)$.
Furthermore, when the multiplicity $m$ 
returns $m(H)=1$ for all $H\in\AAA$, 
then we regard $(\AAA,m)$ simply as the hyperplane arrangement 
$\AAA=(\AAA,1)$. 
In contrast to multiarrangements, hyperplane arrangements are often 
called \emph{simple} arrangements. 
For a multiarrangement $(\AAA,m)$, we define the corresponding 
\emph{logarithmic derivation module} $D(\AAA,m)$ 
as in \Cref{definition:derivation module}, 
which is a graded $S$-submodule of $\Der S$. 
If the module $D(\AAA,m)$ is free as an $S$-module, 
the multiarrangement itself is called \emph{free}. 
When $D(\AAA,m)$ is free, additionally, 
we define the \emph{exponents} $\exp(\AAA,m)$ 
as the multiset of degrees of its homogeneous basis. 
In the theory of hyperplane arrangements, 
the following conjecture by Terao is one of the most important topics 
for decades. 
\begin{conjecture}[Terao's Conjecture]\label{conjecture:Terao}
    The freeness of simple arrangements 
    depends only on their combinatorics.
\end{conjecture}
It is well known that $2$-multiarrangements 
are always free with the logarithmic derivation modules of rank $2$ 
(see \cite{Ziegler}). 
Therefore, \Cref{conjecture:Terao} holds when $\ell=2$. 
On the other hand, $3$-arrangements are not free in general; 
for example, the simple arrangement defined by 
the polynomial
$xyz(x+y+z)$ is not free. 
Therefore, it is natural to study 
the freeness of 
$3$-arrangements as the next step. 
Note that we usually use the coordinates $\{x,y,z\}$ instead of $\{x_1,x_2,x_3\}$ 
when $\ell=3$, and we use 
$\{x,y\}$ 
when $\ell=2$. 
As one of the most important results, in \cite{Yoshinaga-criterion}, 
Yoshinaga gave a criterion to obtain the freeness of simple $3$-arrangements 
by considering certain associated $2$-multiarrangements. 
According to Yoshinaga's criterion, 
the freeness of a $3$-arrangement $\AAA$ 
is characterized by its combinatorics and the exponents 
of a $2$-multiarrangement which is obtained as 
the pair of 
the restriction $\AAA^H$ 
on any plane $H\in\AAA$ and a certain multiplicity studied in Ziegler \cite{Ziegler}. 
Hence, it is so important to understand the exponents of $2$-multiarrangements 
for approaching \Cref{conjecture:Terao} in case $\ell=3$. 

For a $2$-multiarrangement $(\AAA,m)$ 
with $\exp(\AAA,m)=(d_1,d_2)$, define 
the \emph{difference} of the exponents as $\Delta(\AAA,m):=|d_1-d_2|$. 
From the perspective of the difference, 
the most important assumption for multiplicities is the \emph{balanced} property 
(see \Cref{definition:balanced}). 
Roughly speaking, when we increase the multiplicity $m(H)$ by one, step by step 
for a fixed $H\in\AAA$, 
then any multiplicity becomes \emph{unbalanced} and 
the difference of the exponents increases strictly by one 
from a sufficiently large step. 
Completely different from the unbalanced cases, 
in \cite{Wakefield-Yuzvinsky}, Wakefield and Yuzvinsky showed that 
the difference is at most one 
for any balanced multiplicity 
when its underlying arrangement is in the general position 
as in \Cref{corollary:general position}. 
However, it is still considered very difficult to completely classify the exponents. 
Even when the underlying arrangement 
is 
a \emph{Coxeter arrangement}, 
which is defined as the set of all reflection hyperplanes of a Coxeter group 
and whose freeness was proved by Saito in \cite{Saito}, 
the exponents are difficult to classify; 
such classification has been completed 
only for the type $A_2$ 
by Wakamiko in \cite{Wakamiko}. 
This is equivalent to 
the case where its underlying arrangement $\AAA$ consists of three lines: $|\AAA|=3$ (see \Cref{remark:upper bound}). 
On the other hand, 
when we consider the case where $|\AAA|=4$, 
even when $\AAA$ is 
the Coxeter arrangement of type $B_2$, 
only a few parts of the classification of $\exp(\AAA,m)$ 
have been obtained 
(\cite{Wakamiko:universal,2312.06356} and also see \Cref{remark:upper bound}). 
Although the rest of the classification is still considered to be 
too difficult, 
it might be natural to suggest the following conjecture, 
based on some computer calculations. 
\begin{conjecture}[Coxeter $B_2$]\label{conjecture:B_2 exponents large enough}
    For a balanced multiarrangement $(\AAA,m)$ 
    where $Q(\AAA,m)=x^{m_1}y^{m_2}(x-y)^{m_3}(x+y)^{m_4}$ 
    and 
    $m_i\geq3$, 
    define $n_1:=|m_2-m_1|$ and $n_2:=|m_4-m_3|$. 
    Then, 
    there exists a natural number $N\in\NN$ such that 
    $\min\{n_1, n_2\}\geq N$ implies $\Delta(\AAA,m)\leq1$. 
\end{conjecture}

When 
the underlying arrangement $\AAA$ satisfies $|\AAA|\geq5$, 
the classification of exponents seems much harder. 
By considering certain matrices named WY-matrix in \Cref{section:introduction to WY-matrix}, 
which were essentially introduced by Wakefield and Yuzvinsky in \cite{Wakefield-Yuzvinsky}, 
we obtain a new result on exponents for arbitrary $|\AAA|\geq4$ as follows. 

\begin{restatable}{theorem}{maintheorem} 
\label{theorem:main theorem}
    Let $n\geq2$ be an arbitrary natural number. 
    Then, there exists a $2$-multiarrangement $(\AAA,m)$ such that 
    $|\AAA|=n$, 
    $Q(\AAA,m)=x^{m_1}y^{m_2}\prod_{i=3}^{n}(x-s_iy)^{m_i}$ where $s_i\in\ZZ$ 
    and 
    $\Delta(\AAA, m+k\cdot(\delta_{\ker x}+\delta_{\ker y}))=0$ for any $k\in\NN$ 
    (see \Cref{remark:delta_H} for the definition of $\delta$). 
\end{restatable}

This result is in contrast to the properties of unbalanced multiplicities, 
which is mentioned in the last paragraph: 
For any 2-multiarrangement $(\AAA,m)$ and $H\in\AAA$, 
$\Delta(\AAA,m+k\cdot\delta_H)$ strictly increases for sufficiently large $k\in\NN$. 
\begin{example}
    \label{example:45313}
    Let $(\AAA,m)$ be a multiarrangement 
    defined by the polynomial 
    $Q(\AAA,m)=x^{k}y^{k+1}(x-y)^{3}(x-3y)^{1}(x-2y)^{3}$. 
    Then, we have 
    $\Delta(\AAA,m)=2$ if and only if $k=4$, and 
    $\Delta(\AAA,m)=0$ otherwise. 
    See \Cref{example:45313:proof} for the proof. 
\end{example}

\begin{remark}
When we assume that $s_i$ are algebraically independent, instead of integers, 
then \Cref{theorem:main theorem} is trivial by \Cref{theorem:main:Wakefield--Yuzvinsky}. 
\end{remark}
For a fixed underlying $2$-arrangement, 
the analysis of the distribution 
$\Delta=0$ is essential for the complete classification of exponents, 
according to the theory of multiplicity lattice 
by Abe and Numata \cite{Abe-Numata}. 
As similarly to 
\Cref{theorem:main theorem}, 
we also obtain the following result 
considering the case $n=4$, $s_3=1$, $s_4=2$ 
and a change of coordinates in \Cref{remark:transformation}. 
\begin{theorem}[Coxeter $B_2$ -(i)]
\label{corollary:main theorem}
Let $(\AAA,m)$ be a multiarrangement defined by $Q(\AAA,m)=x^{m_1}y^{m_2}(x-y)^{m_3}(x+y)^{m_4}$. 
If $|{m_2}-{m_1}|=|{m_4}-{m_3}|\geq 1$, then 
we have $\Delta(\AAA,m)=0$. 
\end{theorem}
\begin{example}
    \label{example:corollary to B2}
Let  $(\AAA,m)$ be a multiarrangement defined by 
$Q(\AAA,m)=x^{5}y^{9}(x-y)^{3}(x+y)^{7}$. 
Since $|m_2-m_1|=|m_4-m_3|=4$, 
we have $\Delta(\AAA,m)=0$, or $\exp(\AAA,m)=(12,12)$.
\end{example}
In other words, \Cref{corollary:main theorem} 
gives a partial answer to \Cref{conjecture:B_2 exponents large enough}, 
just in case $n_1=n_2$. 
Note that for the case $n_1=n_2=0$, 
the exponents were classified by Wakamiko \cite{Wakamiko:universal}. 

On the other hand, 
when $|m_2-m_1|>|m_4-m_3|$, 
we have the following. 
\begin{theorem}[Coxeter $B_2$ -(ii)]
\label{theorem:B2exponents_new}
Let $(\AAA,m)$ be a multiarrangement defined by 
$Q(\AAA,m)=x^{m_1}y^{m_2}(x-y)^{m_3}(x+y)^{m_4}$. 
If $({m_2}-{m_1})=({m_4}-{m_3})+2k$ for $k\in\NN$ 
and 
$\det (a_{i,j})_{1\leq i,j\leq k}\in2\ZZ+1$  where 
$a_{i,j}=\binom{\frac{|m|}{2}-m_1-j}{\frac{|m|}{2}-m_2-1+i}$, 
then we have $\Delta(\AAA,m)=0$. 
\end{theorem}

See \Cref{corollary:thm1.8} and \Cref{example:thm1.8} for applications of this theorem. 
As another application of WY-matrix, 
we also construct a $2$-arrangement $\AAA$ 
which does not admit any balanced multiplicity $m$ satisfying $\Delta(\AAA,m)=|\AAA|-2$, 
except for the simple case $m\equiv1$ (cf. \Cref{theorem:upper bound of delta}). 
%
%
%

\begin{theorem}
\label{theorem: no peak points}
Let $(\AAA,m)$ be a balanced multiarrangement defined by 
$x^{m_1}y^{m_2}(x-z_3y)^{m_3}(x-z_4y)^{m_4}$ 
where $z_3$ is algebraic and $z_4$ is a transcendental number. 
If $m\not\equiv 1$, then $\Delta(\AAA,m)\leq1$. 
\end{theorem}

The rest of this article is composed as follows. 
In \Cref{section:preliminary}, we prepare some basic results 
about arrangement theory. 
In \Cref{section:introduction to WY-matrix}, we define the WY-matrix, 
and also introduce some notation around the matrix. 
Finally, in \Cref{section:proof of main theorem}, we show the proofs of our main results. 
%
\section{Preliminaries}\label{section:preliminary}
In this section, we prepare some basic results that are 
necessary for our results. 
First, we recall the definition of the logarithmic derivation module 
as follows. 
\begin{definition}\label{definition:derivation module}
    For an $\ell$-multiarrangement $(\AAA,m)$, we define the 
    \emph{logarithmic derivation module $D(\AAA,m)$}, 
    a submodule of $\Der S$, by 
    $\{\theta \in \Der S \mid \theta(\alpha_{H})\in\alpha_{H}^{m(H)}S 
    \text{ for all } H\in\AAA \}$. 
    Note that $D(\AAA,m)$ is a graded $S$-module 
    $D(\AAA,m)=\bigoplus_{d\geq0} D(\AAA,m)_d$. 
    When $D(\AAA,m)$ is generated by homogeneous derivations 
    $\{\theta_i \in D(\AAA,m) \mid 1 \leq i \leq \ell\}$ 
    that are linearly independent over $S$, we call 
    $(\AAA,m)$ \emph{free} with \emph{exponents} 
    $\exp(\AAA,m):=(\deg(\theta_1),\deg(\theta_2),\ldots,\deg(\theta_{\ell}) )$. 
\end{definition}

The following result is useful to check the freeness of multiarrangements 
as long as we can obtain a candidate of a basis for the logarithmic derivation module. 
\begin{theorem}[Saito's criterion, \cite{Saito,Ziegler}]
\label{theorem:Saito's criterion} 
    For an $\ell$-multiarrangement $(\AAA,m)$ 
    and a set of homogeneous derivations 
    $\{\theta_{i} \mid 1 \leq i \leq \ell\} \subset D(\AAA,m)$, 
    this set of derivations forms a basis of $D(\AAA,m)$ if and only if 
    $\det((\theta_i(x_j))_{1 \leq i,j \leq \ell})=c\cdot Q(\AAA,m)$ 
    for some $c\in\KK\setminus\{0\}$.
    In particular, it follows that 
    $\sum_{1 \leq i \leq \ell}\deg(\theta_i)=|m|=\sum_{H\in\AAA}m(H)$ 
    for a free multiarrangement and its homogeneous basis.
\end{theorem}
Throughout the rest of this article, 
we assume that $\AAA$ is a $2$-arrangement unless otherwise specified. 
Then, we especially use \Cref{theorem:Saito's criterion} 
in the statement that 
$\deg(\theta_1)+\deg(\theta_2)=|m|$ with a homogeneous basis $\{\theta_1,\theta_2\}$ 
for $D(\AAA,m)$. 
When the multiplicity $m$ 
has a special condition called unbalanced, 
Wakefield and Yuzvinsky determined the exponents 
as follows. 
%
\begin{definition}\label{definition:balanced} 
    When a multiplicity $m$ on a $2$-arrangement $\AAA$ satisfies 
    $m(H)<|m|/2$ for all $H\in\AAA$, 
    we call the multiplicity (and the multiarrangement itself) \emph{balanced}. 
    Otherwise, 
    we call 
    them \emph{unbalanced}. 
\end{definition}
\begin{theorem}[Case 2.1, \cite{Wakefield-Yuzvinsky}]\label{theorem:unbalanced}
    If a $2$-dimensional multiarrangement $(\AAA,m)$ is unbalanced 
    with $m(H)\geq|m|/2$ for some $H\in\AAA$, 
    then we have $\exp(\AAA,m)=(m(H),|m|-m(H))$. 
\end{theorem}
    On the other hand, to classify the exponents for 
    balanced multiarrangements 
    is much more difficult. 
    As one of the most important results, in \cite{Abe:chamber}, 
    Abe showed an upper bound for the difference of the exponents $\Delta$ as follows.
\begin{theorem}[Theorem 1.6, \cite{Abe:chamber}]\label{theorem:upper bound of delta}
    For a $2$-dimensional balanced multiarrangement $(\AAA,m)$ 
    with $\exp(\AAA,m)=(d_1,d_2)$,  
    we have $\Delta(\AAA,m)=|d_1-d_2|\leq|\AAA|-2$. 
\end{theorem}
\begin{remark}\label{remark:upper bound}
    \Cref{theorem:upper bound of delta} includes the following result, 
    which was solved for the first time by Wakamiko in \cite{Wakamiko}: 
    For a balanced multiarrangement $(\AAA,m)$ defined by $x^{m_1}y^{m_2}(x-y)^{m_3}$, 
    we have $\Delta(\AAA,m)\leq1$. 
    Similarly, for a balanced multiarrangement $(\AAA,m)$ satisfying $|\AAA|=4$, 
    \Cref{theorem:upper bound of delta} also gives $\Delta(\AAA,m)\in\{0,1,2\}$. 
    Combined with \Cref{theorem:Saito's criterion}, it quickly turns out that 
    $\Delta(\AAA,m)=1$ if and only if $|m|\in2\ZZ+1$, 
    while it is very difficult to classify the exponents when $|m|\in2\ZZ$. 
    As a challenge to this problem, in \cite{2312.06356}, 
    the author and Numata proved a conjecture 
    which had been given for sufficiently large multiplicities 
    by Abe in \cite{Abe:B_2}. 
    We obtain another proof for both of these facts 
    in \Cref{corollary: B_2 exponents in MN,corollary:A_2 exponents}, 
    considering the determinants of certain matrices introduced in 
    \Cref{section:introduction to WY-matrix}. 
\end{remark}
    When we consider the constant one multiplicity $m\equiv1$, 
    it is well-known that $\exp(\AAA,m)=\exp(\AAA)=(1,|\AAA|-1)$ since 
    there always exists the \emph{Euler derivation} 
    $\theta_{E}:=x\partial_x+y\partial_y\in D(\AAA)$ 
    (for example, see \cite{Orlik-Terao}). 
    Hence, simple arrangements admit the upper bound, 
    $\Delta(\AAA)=|\AAA|-2$. 
    Note that $\theta_0:=\frac{Q(\AAA,m)}{Q(\AAA)}\cdot\theta_E\in D(\AAA,m)$ 
    is the minimal degree derivation of the form $f\cdot\theta_E$ 
    with $f\in S\setminus\{0\}$. 
    This derivation $\theta_0$ is (one of) the lower basis for $D(\AAA,m)$ 
    if and only if $|m|\leq2|\AAA|-2$ as in \Cref{theorem:small}; 
    this type of $2$-multiarrangements $(\AAA,m)$ 
    satisfy 
    $\Delta\leq|\AAA|-3$, except for the case 
    $m\equiv1$. 
    The Euler derivation $\theta_E$ and the constant multiplicity $m\equiv1$ can be considered special in this sense. 
\begin{theorem}[Case 2.2, \cite{Wakefield-Yuzvinsky}]\label{theorem:small}
    If a $2$-multiarrangement $(\AAA,m)$ satisfies 
    $|m|\leq2|\AAA|-2$, 
    we have $\exp(\AAA,m)=(|m|-|\AAA|+1,|\AAA|-1)$. 
\end{theorem}
    %
%
%
\begin{remark}\label{remark:transformation}
When we study the exponents of multiarrangements, 
it is sometimes useful to consider some coordinate changes; 
multiarrangements which are transformed to each other 
by an invertible matrix have the same exponents. 
For example, define a matrix $T$ as follows. 
\[
T:=
\left(
\begin{array}{cc}
 1 & 1 \\
 0 & 1
\end{array}
\right),\
T^{-1}=
\left(
\begin{array}{cc}
 1 &-1 \\
 0 & 1
\end{array}
\right).
\]
Then, considering the transformation $\cevxx{X}{Y}^\top=T\cevxx{x}{y}^\top$, 
the multiarrangement $(\AAA,m)$ defined by $Q(\AAA,m)=x^ay^b(x-y)^c(x+y)^d$ 
has the same exponents as 
the multiarrangement $T(\AAA,m)$, defined by $Q(T(\AAA,m)):=(X-Y)^aY^b(X-2Y)^cX^d$. 
We introduce an equivalence relation on the set of multiarrangements 
by such transformations as $(\AAA,m)\sim T(\AAA,m)$. 
More precisely, 
for an arbitrary $\theta=f(x,y)\partial_x + g(x,y)\partial_y\in D(\AAA,m)$, 
it follows that 
$$F(X,Y)\partial_X+G(X,Y)\partial_Y:=\cevxx{\partial_X}{\partial_Y}T^{-1}\cevxx{f(x,y)}{g(x,y)}^\top \in D(T(\AAA,m)).$$
\end{remark}

\begin{remark}\label{remark:delta_H}
    For a $2$-arrangement $\AAA$ and each $H\in\AAA$, 
    let us define the multiplicity $\delta_H$ on $\AAA$ 
    as follows. 
\begin{align*}
\delta_H(H'):=
\left\{
\begin{array}{ll}
1 & (H'=H)\\
0 & (H'\neq H)
\end{array}
\right.
\end{align*}
    Then, we have $|\Delta(\AAA,m)-\Delta(\AAA,m+\delta_H)|=1$ 
    for any multiarrangement. 
    In particular, if a multiplicity $m$ satisfies $\Delta(\AAA,m)=0$, 
    it follows that $\Delta(\AAA,m\pm\delta_H)=1$ for arbitrary $H\in\AAA$.     
    See \cite{Abe-Numata} 
    for further behaviors of $\Delta$ and $\delta_H$. 
\end{remark}


\section{Introduction to WY-matrix}\label{section:introduction to WY-matrix}

\subsection{Definition of the WY-matrix}\label{subsection:1}
For a balanced $2$-multiarrangement $(\AAA,m)$ and any $e\in\NN$, 
we define the \emph{Wakefield--Yuzvinsky matrix}, or \emph{WY-matrix} for short, 
$M_{WY}(\AAA,m;e)$ in the research on the exponents. 
Since such a matrix was studied in \cite{Wakefield-Yuzvinsky}, 
we name it after the authors' initials. 
%
Let us assume that $(\AAA,m)$ is defined by the defining polynomial 
$Q(\AAA,m)=x^{m_1}y^{m_2}\prod_{i=3}^{n}(x-s_iy)^{m_i}$. 
Then, 
the construction of WY-matrix gives us a matrix method for 
the existence (or non-existence) of 
$\theta\in D(\AAA,m)\setminus\{0\}$ of degree $e$ by the following four steps: 
\begin{enumerate}\setlength{\leftskip}{6mm}
    \setlength{\parskip}{2mm} 
    \setlength{\itemsep}{0mm} 
        \item[\textbf{(Step 1)}] 
        Assume that there exists a homogeneous derivation 
        $\theta=x^{m_1}f(x,y)\partial_x-y^{m_2}g(x,y)\partial_y\in D(\AAA,m)_e$, where 
        $$
        f(x,y)=\sum_{j=0}^{e-m_1}f_jx^{e-m_1-j}y^{j} \ \ 
        \text{and} \ \
        g(x,y)=\sum_{j=0}^{e-m_2}g_jx^{e-m_2-j}y^{j}.$$
        \item[\textbf{(Step 2)}] 
        Recall that $\theta(x-s_iy)\in(x-s_iy)^{m_i}S$ for $3 \leq i \leq n$. 
        \item[\textbf{(Step 3)}] 
        For each $i$, 
        the condition in (Step 2) can be expressed as 
        $(\frac{\partial}{\partial x})^k(\theta(x-s_iy))|_{(x,y)=(s_i,1)}=0$ 
        for $0\leq k \leq m_i-1$. 
        \item[\textbf{(Step 4)}] 
        As the matrix of the coefficients for $\{f_j\}$ and $\{g_j\}$ 
        in (Step 3), 
        we obtain a system of equations 
        $M (f_0 \ f_1 \ \cdots \ f_{e-m_1} \ 
        g_0 \ g_1 \ \cdots \ g_{e-m_2})^{\top}=O$, where 
        the matrix $M$ has 
        $(\sum_{i=3}^{n}m_i)$-rows and 
        $((e-m_1+1)+(e-m_2+1))$-columns. 
        Call the left $(e-m_1+1)$-columns $f$-part and 
        the right $(e-m_2+1)$-columns $g$-part. 
\end{enumerate}
\begin{definition}
\label{definition:WY-matrix}
For a balanced $2$-multiarrangement $(\AAA,m)$ and arbitrary $e\in\NN$, 
we call $M_{WY}(\AAA,m; e):=M$, 
the matrix constructed through the previous four steps, 
\emph{WY-matrix}. 
Moreover, 
when we consider the case $e=\frac{|m|}{2}-1$ for a multiarrangement $(\AAA,m)$ of $|m|\in2\ZZ$, 
let $M_{WY}(\AAA,m)$ simply denote $M_{WY}(\AAA,m;\frac{|m|}{2}-1)$ 
in the rest of this article (see \Cref{remark:square case} for the reason). 
\end{definition}
%
\begin{theorem}
\label{theorem:WY-matrix}
For a balanced $2$-multiarrangement $(\AAA,m)$, 
there exists a nonzero derivation 
$\theta\in D(\AAA,m)$ of degree $e$ 
if and only if the WY-matrix 
$M_{WY}(\AAA,m; e)$ does not have full rank. 
\end{theorem}
\begin{remark}\label{remark:elements of WY-matrix}
Explicitly writing, the equations in (Step 3) above are 
$$\sum_{j=0}^{e-m_1}f_j\frac{(e-j)!}{(e-j-k)!}s_i^{e-j-k} 
+ s_i\cdot \sum_{j=0}^{e-m_2}g_j\frac{(e-m_2-j)!}{(e-m_2-j-k)!}s_i^{e-m_2-j-k}=0,$$ 
where $s_i^r=0$ for $r<0$ (see Section 4 in \cite{Wakefield-Yuzvinsky}). 
\end{remark}
%
\begin{example}\label{example:M_WY:2213}
    For the multiarrangement $(\AAA,m)$ defined by 
    $Q(\AAA,m)=x^2y^2(x-y)^1(x-{s_4} y)^3$ where $s_4\in\KK\setminus\{0,1\}$, 
    we show that $\exp(\AAA,m)=(3,5)$ if and only if $s_4=-1$. 
    Let us assume that there exists a derivation $\theta\in D(\AAA,m)$ with degree $3$ 
    such that 
    $\theta
    =x^2(f_0x+f_1y) \partial_x - y^2(g_0x+g_1y) \partial_y
    \neq0$ for some $f_i,g_i\in\KK$. 
    By \Cref{definition:derivation module}, we have the following: 
    \begin{align*}
    \theta(x-y)&=x^2(f_0x+f_1y) + y^2(g_0x+g_1y)\in(x-y)S \hspace{9.9mm}\cdots (a)\\ 
    \theta(x-{s_4} y)&=x^2(f_0x+f_1y) + {s_4} y^2(g_0x+g_1y)\in(x-{s_4} y)^3S\ \cdots (b)
    \end{align*}
    This is equivalent to the following four conditions 
    being satisfied at the same time (Step 3): 
    \begin{enumerate}
    \item[($a$\ -\ i)] For $(a)$, substitute 
    $(x,y)=(1,1)$. Then, we have \underline{$(f_0+f_1) + (g_0+g_1)=0$\ }. 
    \item[($b$\ -\ i)] For $(b)$, substitute $(x,y)=({s_4},1)$. 
    Then, we have \underline{${s_4}^2({s_4} f_0+f_1) + {s_4}({s_4} g_0+g_1)=0$\ }. 
    \item[($b$\ -ii)] For $(b)$, 
    partially differentiate with $\frac{\partial}{\partial x}$ 
    and substitute $(x,y)=({s_4},1)$. 
    Then, we have 
    \underline{$(3{s_4}^2 f_0+2{s_4} f_1) + {s_4} (g_0+0\cdot g_1)=0$\ }. 
    \item[($b$-iii)] For $(b)$, 
    partially differentiate with $\frac{\partial^2}{\partial x^2}$ 
    and substitute $(x,y)=({s_4},1)$. 
    Then, we have 
    \underline{$(6{s_4} f_0+2f_1) + {s_4} (0\cdot g_0+0\cdot g_1)=0$\ }. 
    \end{enumerate}
    
    Hence, $\theta\in D(\AAA,m)_3\setminus\{0\}$ 
    if and only if the following system of equations 
    \[
    \left(
    \begin{array}{cc|cc}
     1 & 1 & 1 & 1 \\
     \hline
     {s_4}^3 & {s_4}^2 & {s_4}^2 & {s_4}^1 \\
     3{s_4}^2 & 2{s_4}^1 & {s_4}^1 & 0 \\
     6{s_4}^1 & 2{s_4}^0 & 0 & 0 
    \end{array}
    \right)
    \vecxxxx{f_0}{f_1}{g_0}{g_1}
    =
    \vecxxxx{0}{0}{0}{0}
    \]
    has a non-trivial solution. Let $M_{WY}(\AAA,m;3)$ denote the coefficient matrix above. Since 
    $\det M_{WY}(\AAA,m;3)=2s_4^2(s_4-1)(s_4+1)$, 
    we obtain $\exp(\AAA,m)=(3,5)$ if and only if $s_4=-1$.
\end{example}

\begin{remark}\label{remark:square case}
    When we assume that $|m|\in2\ZZ$ and $e=\frac{|m|}{2}-1$, 
    the matrix $M_{WY}(\AAA,m;e)$ is square since 
    $(e-m_1+1)+(e-m_2+1)=2e+2-(m_1+m_2)=\sum_{i=3}^{n}m_i$. 
    Only these cases are considered in \cite{Wakefield-Yuzvinsky} 
    since they are important in the sense that 
    there exists a derivation $\theta\in D(\AAA,m)$ with degree $\frac{|m|}{2}-1$ 
    if and only if $\Delta(\AAA,m)\geq2$ (see \Cref{corollary:general position}). 
    For this reason, 
    we let $M_{WY}(\AAA,m)$ simply denote $M_{WY}(\AAA,m;\frac{|m|}{2}-1)$ 
    as mentioned in \Cref{definition:WY-matrix}. 
\end{remark}

In order to examine the determinants, we recall a generalized Laplace expansion formula. 
Let $M$ be a square matrix of size $m$ and let 
$\gamma=(\gamma_1,\ldots,\gamma_k)$ be an ordered partition. 
Note that $\sum_{i=1}^{k}\gamma_i=m$ 
and we have $2^{m-1}$ choice of $m$-ordered partitions. 
We partition the rows of this $m\times m$ matrix $M$ into $k$ ordered blocks by this $\gamma$. 
Also, we partition the columns to $k$ blocks by a function 
$\beta:\{1,2,\ldots,m\}\to\{1,2,\ldots,k\}$ 
satisfying 
$\#\beta^{-1}\{i\}=\gamma_i$. 
We call such an assigning function $\beta$ 
a \emph{column assignment for $\gamma$}. 
Let $(\gamma,\beta;i)$ denote 
the square submatrix formed by the intersection of $i$-th row block with columns $\beta^{-1}\{i\}$. 
Since we have $\binom{m}{\gamma_1,\gamma_2,\ldots,\gamma_k}:=\prod_{1\leq i \leq k} \binom{m-\sum_{j<i}\gamma_j}{\gamma_i}$ choice 
for the column assignments, let $\binom{m}{\gamma}$ denote the set of all column assignments for $\gamma$. 
Then, taking all possible column assignmens $\beta\in \binom{m}{\gamma}$ for the fixed $\gamma$, we can compute the determinant as follows. 
%

\begin{formula}[\cite{Wakefield-Yuzvinsky}]
\label{formula:Muir}
For a square matrix $M$ of size $m$, 
fix an ordered partition $\gamma$ for the row. 
Then, it follows that 
$$\det M = \sum_{\beta\in\binom{m}{\gamma}} \sign(\beta)\cdot
\prod_{i=1}^{k} \det(\gamma,\beta;i) 
\text{ where $\sign(\beta)$ is some signum $\pm1$. }$$ 
\end{formula}

Wakefield and Yuzvinsky utilized this formula 
in their paper \cite{Wakefield-Yuzvinsky}. 
See \cite{Wakefield-Yuzvinsky} for more details on the signum. 
This formula plays an important role 
in \Cref{section:proof of main theorem}. 

\begin{remark}[\cite{Wakefield-Yuzvinsky}]\label{remark:partition}
For the multiarrangement defined by 
$Q(\AAA,m)=x^{m_1}y^{m_2}\prod_{i=3}^{n}(x-z_iy)^{m_i}$ where $|m|\in2\ZZ$, 
define $e:=\frac{|m|}{2}-1$. 
By the ordered partition $\gamma=(m_3,m_4,\ldots,m_n)$, 
the square matrix $M_{WY}$ is partitioned into 
$n-2$ row blocks $L_3,L_4,\ldots,L_n$, 
where  
each $L_i$ 
does not contain $z_j$ for $j\neq i$ 
(see (Step 3)). 
%
%
%
\end{remark}
Under the condition in \Cref{remark:partition}, 
$\det M_{WY}$ is a polynomial in $\ZZ[z_3,z_4,\ldots,z_n]$. 
Hence, 
``the complement to the zero locus of $d$ is an open set in the Zariski topology'' 
in $\KK^{n-2}$ (written just before Theorem 4.1 in \cite{Wakefield-Yuzvinsky}, 
where $d=\det M_{WY}$). 
Wakefield and Yuzvinsky used the word ``general position'' as in 
\Cref{corollary:general position} in this sense. 

\begin{theorem}[Theorem 4.1, \cite{Wakefield-Yuzvinsky}]
\label{theorem:main:Wakefield--Yuzvinsky}
    Let $(\AAA,m)$ be a balanced multiarrangement defined by 
    $Q(\AAA,m)=x^{m_1}y^{m_2}\prod_{i=3}^{n}(x-z_iy)^{m_i}$ 
    where $|m|\in2\ZZ$ and $|m|\geq2n-2$. 
    Then, $\det M_{WY}(\AAA,m)$ is not identically zero as a polynomial in $\ZZ[z_3,\ldots,z_n]$.
\end{theorem}
%
\begin{corollary}[Theorem 3.1, \cite{Wakefield-Yuzvinsky}]
\label{corollary:general position}
Suppose $m$ is balanced and $|m|\geq2n-2$. 
Then, there exists a general position of $n$ lines such that 
$\Delta(\AAA,m)\leq1$ 
for every multiarrangement $(\AAA,m)$ of lines in this position, 
having $m$ as the multiplicity. 
\end{corollary}
As another application of \Cref{theorem:main:Wakefield--Yuzvinsky}, 
we also have \Cref{theorem: no peak points} as follows. 

\begin{proof}[Proof of \Cref{theorem: no peak points}]
    When $|m|\leq6$, the result follows from \Cref{theorem:small}. 
    Now assume $|m|\geq7$. 
    As mentioned in \Cref{remark:delta_H}, 
    it is sufficient to consider the case $|m|\in2\ZZ$. 
    For a fixed multiplicity $m$, 
    assume there exists a pair of 
    an algebraic number $z_3=s_3$ and a transcendental number $z_4=s_4$ 
    such that $\Delta(\AAA({s_3,s_4}),m)\geq2$, 
    where $Q(\AAA({s_3,s_4})):=xy(x-s_3y)(x-s_4y)$. 
    Since $\det M_{WY}(\AAA,m)\in \ZZ[z_3,z_4]\setminus\{0\}$ 
    by \Cref{theorem:main:Wakefield--Yuzvinsky}
    and $\det M_{WY}(\AAA({s_3,s_4}),m)=0$, 
    we have $\det M_{WY}(\AAA,m)\in(z_3-s_3)\KK[z_3,z_4]$ 
    (otherwise, it would contradict the assumption that $s_4$ is a transcendental number). 
    Hence, it follows that $\det M_{WY}(\AAA({s_3,s_4'}),m)=0$ 
    for any other $z_4=s_4'$. 
    For any nonzero algebraic number $a$, define $s_4':=as_4$. 
    Note that 
    the arrangement $\AAA({s_3,s_4'})$ is equivalent to 
    the one defined by $xy(x-\frac{s_3}{a}y)(x-s_4y)$. 
    Since $\det M_{WY}(\AAA,m)\in \ZZ[z_3,z_4]\setminus\{0\}$ and 
    $\det M_{WY}(\AAA(\frac{s_3}{a},s_4),m)=0$, 
    we also obtain $\det M_{WY}(\AAA,m)\in(z_3-\frac{s_3}{a})\KK[z_3,z_4]$, 
    which is a contradiction. 
    \end{proof}
%

\begin{corollary}\label{corollary: no peak points_pi}
Let $(\AAA,m)$ be the multiarrangement defined by 
$x^{m_1}y^{m_2}(x-y)^{m_3}(x-\pi y)^{m_4}$. 
If $m$ is balanced and $m\not\equiv 1$, then $\Delta(\AAA,m)\leq1$. 
\end{corollary}

\begin{remark}
Note that every 
$2$-arrangement $\AAA$ consisting of four lines 
can be written by a defining polynomial 
$Q(\AAA)\sim XY(X-Y)(X-sY)$ with $s\in\KK$ under some 
change of coordinates 
as in \Cref{remark:transformation}. 
Let us define a new class of $3$-arrangements $\mathcal{T}_\mathcal{R}$ as follows: 
$\mathcal{T}_\mathcal{R}:=\{\AAA \text{: $3$-arrangement} \mid 
\exists H \in \AAA\ \suchthat\ Q(\AAA^H)\sim XY(X-Y)(X-sY) 
\text{ where 
$s$ is a transcendental number}\}$. 
Then, \Cref{theorem: no peak points} and Yoshinaga's criterion imply that 
    $\mathcal{T}_\mathcal{R}$ is a class of $3$-arrangements 
    where \Cref{conjecture:Terao} is true. 
\end{remark}
\begin{question}
When we assume that $\{s_3,s_4\}$ are algebraically independent, we also have $\Delta(\AAA(s_3,s_4),m)\leq1$ for any balanced multiplicity $m\not\equiv1$. 
Then, 
it is natural to ask whether there exists an algebraic number \hspace{-1.5mm} 
$s\neq0,1$ \hspace{-1.5mm}
such \hspace{-1.5mm} that \hspace{-0.7mm} 
$\Delta(\AAA(1,s),m)\hspace{-1.0mm}\leq\hspace{-1.0mm}1$ \hspace{-1.3mm}
for any balanced \hspace{-1.7mm} 
$m\hspace{-1.0mm}\not\equiv\hspace{-1.0mm} 1$. 
\end{question}

\subsection{Tuple}\label{subsection:tuple}
In this subsection, we fix some useful notation around tuple. 
\begin{definition}\label{definition:tuple}
Let $a=(a_i \mid n \geq i \geq 0)=(a_n,a_{n-1},\ldots,a_1,a_0)$ 
be a (finite) tuple of nonnegative integers. 
Throughout this article, we consider finite tuples only. 
\begin{itemize}
\item If $a_{i+1} > a_i$ for all $i$, we say that $a$ is (strictly) \emph{descending}. 
Additionally, 
if $a_{i+1}=a_i+1$ for every $i$, we say that $a$ is \emph{continuous}.
\item For a tuple $a=(a_i \mid n \geq i \geq 0)$, 
let us define the \emph{length} of $a$ by $\length{a}:=n+1$. 
\item For nonnegative integers $n_2>n_1\geq0$, 
let $\tuplee{n_2}{n_1}$ denote the continuous tuple of length $n_2-n_1+1$, 
$\tuplee{n_2}{n_1}:=(n_2,n_2-1,\ldots,n_1+1,n_1)$. 
When there is no 
danger of confusion, 
we often write $(n_2:n_1)$ to mean $\tuplee{n_2}{n_1}$.
\item For tuples $a=(a_i \mid n \geq i \geq 0)$ and $a'=(a'_j \mid n' \geq j \geq 0)$, 
let $a \oplus a'$ denote a new tuple of length $(n+n'+2)$ as 
$a \oplus a':=(\underbrace{a_{n}, a_{n-1}, \ldots, a_{1}, a_{0}}_{=a}, 
\underbrace{a'_{n'}, a'_{n'-1}, \ldots, a'_{1}, a'_{0}}_{=a'})$. \vspace{-3mm}
\item For a tuple $a=(a_i \mid n \geq i \geq 0)$, we define the \emph{even-part} by $\even{a}:=(a_i\in a \mid a_i\in2\ZZ)$ and the \emph{odd-part} by 
$\odd{a}:=(a_i\in a \mid a_i\in2\ZZ+1)$.
\end{itemize}
\end{definition}
\begin{example} 
$a=(5,4,2,0)$ is descending but not continuous, while 
$a'=(3,2,1,0)$ is not only descending but also continuous. 
It is sometimes useful to divide $a,a'$ into even and odd parts as $\even{a}=(4,2,0),\odd{a}=(5)$ and $\even{a'}=(2,0),\odd{a'}=(3,1)$. We can also express $a'=\tuple{3}$, $\even{a}=\tuple{4}_{\text{even}}$.
When we regard a tuple $a$ as a set of integers by using the symbol $\text{Set}(a)$, 
it follows that 
$\text{Set}(a)=\text{Set}(\even{a})\sqcup\text{Set}(\odd{a})$. 
\end{example}
%


\subsection{Reformulation of WY-matrix }\label{subsection:component_F}
We define two components $P(\AAA,m)$ and $W(\AAA,m)$ 
for the WY-matrix $M_{WY}(\AAA,m)$ under the Hadamard (entrywise) product. 
Let 
$\varphi_{a_i}(t):=t^{a_i}$ denote 
the power function with respect to a variable $t$ 
for every 
integer $a_i\in \ZZ$. 
As Wronskian is intensively utilized in \cite{Wakefield-Yuzvinsky}, 
the WY-matrix is mostly written by Wronski matrix. 

\begin{definition} 
    For a tuple of nonnegative integers $a=(a_n,a_{n-1},\ldots,a_0)$ and $k\in\NN$, 
    let us define the \emph{Wronski matrix} of $k$-rows and $(n+1)$-columns 
    with respect to a variable $t$ by 
\[
W_t(k;a):=
\left(
\begin{array}{cccc}
 \varphi_{a_n}(t) & \varphi_{a_{n-1}}(t) & \cdots & \varphi_{a_{0}}(t) \\
 \varphi_{a_n}^{(1)}(t) & \varphi_{a_{n-1}}^{(1)}(t) & \cdots & \varphi_{a_{0}}^{(1)}(t) \\
 \vdots & \vdots & \ddots & \vdots \\
 \varphi_{a_n}^{(k-1)}(t) & \varphi_{a_{n-1}}^{(k-1)}(t) & \cdots & \varphi_{a_{0}}^{(k-1)}(t) \\\end{array}
\right),
\]
and define $W(k;a):=W_{t}(k;a)|_{t=1}$. 
\end{definition}

\begin{example}
\[
W_t(3;(3,2))=
\left(
\begin{array}{cc}
 t^3 & t^2 \\
 3t^2 & 2t \\
 6t & 2 
\end{array}
\right),\ 
W_t(3;(1,0))=
\left(
\begin{array}{cc}
 t^1 & t^0 \\
 1 & 0 \\
 0 & 0 
\end{array}
\right),\ 
W(3;(3,2))=
\left(
\begin{array}{cc}
 1 & 1 \\
 3 & 2 \\
 6 & 2 
\end{array}
\right).
\]
\end{example}

Then, we have the following. 
\begin{remark}\label{remark:WY-matrix_2} 
For a balanced $2$-multiarrangement $(\AAA,m)$ defined by 
$$Q(\AAA,m)=x^{m_1}y^{m_2}\cdot\prod_{i=3}^{n}(x-z_iy)^{m_i}\ \text{where } z_i\in\KK\setminus\{0\}\ \text{and } |m|\in2\ZZ,$$
let $e:=\frac{|m|}{2}-1$. Then, we have 
\[
M_{WY}(\AAA,m)=
\left(
\begin{array}{c|c}
 W_{z_3}(m_3;(e:m_1)^{\text{tuple}}) & z_3\cdot W_{z_3}(m_3;(e-m_2:0)^{\text{tuple}}) \\
\hline
 W_{z_4}(m_4;(e:m_1)^{\text{tuple}}) & z_4\cdot W_{z_4}(m_4;(e-m_2:0)^{\text{tuple}}) \\
 \hline
 \vdots & \vdots \\
 \hline
 W_{z_n}(m_n;(e:m_1)^{\text{tuple}}) & z_n\cdot W_{z_n}(m_n;(e-m_2:0)^{\text{tuple}}) \\
\end{array}
\right).
\]
\end{remark}

When we consider the case $k=\length{a}$, then 
$W(k;a)=W(\length{a};a)$ is a square matrix of size $\length{a}$. 
In this case, the determinant called \emph{Wronskian} 
$w(a):= \det W(\length{a};a)$ 
was formulated by Wakefield and Yuzvinsky as follows. 
\begin{theorem}[Lemma 4.2, \cite{Wakefield-Yuzvinsky}]\label{theorem:Wronskian}
    When a tuple of nonnegative integers $a=(a_n,a_{n-1},\ldots,a_0)$ satisfies $a_{k}>a_{k-1}$ for all $k$ between $n \geq k \geq 1$, we have 
    $$w(a)= \det W(\length{a};a)=(-1)^{\left\lfloor\frac{n+1}{2}\right\rfloor}\cdot\prod_{0 \leq i < j \leq n}(a_i-a_j)\neq0.$$
\end{theorem}

\begin{remark}
    By \Cref{theorem:Wronskian}, it follows that 
    $w(a_i\mid n \geq i \geq 0 )=w(a_i - a_0\mid n \geq i \geq 0 )$ 
    for any descending tuple $a$. 
    Hence, we often assume $a_0=0$ when we focus on the Wronskian.     
\end{remark}
\begin{example}
    For $n\in\NN$, let $a=(a_i\mid n \geq i \geq 0)$ be a continuous tuple of length $n+1$. Then, we have $w(a)=w(\tuple{n})
    =(-1)^{\left\lfloor\frac{n+1}{2}\right\rfloor}\cdot\prod_{0 \leq k \leq n}k!$. 
\end{example}

Utilizing $WY$-matrix, 
we obtain another proof for 
some known results 
(see \Cref{remark:upper bound}). 
\begin{corollary}[Theorem 1.5, \cite{Wakamiko}]
\label{corollary:A_2 exponents}
For a balanced multiarrangement $(\AAA,m)$ defined by 
$Q(\AAA,m)=x^{m_1}y^{m_2}(x-y)^{m_3}$, we have $\Delta(\AAA,m)\leq1$. 
\end{corollary}
\begin{proof}
We may assume that $|m|\in2\ZZ$ by \Cref{remark:delta_H}. 
Define $e:=\frac{|m|}{2}-1$. 
Since $m$ is balanced, we can also assume that $m_1+m_2>m_3$ without loss of generality 
(see \Cref{remark:transformation}). 
Thus, it follows that 
$m_1>e-m_2$ since 
$2(m_1+m_2)-2e=m_1+m_2-m_3+2>0$.
By \Cref{remark:WY-matrix_2}, we have 
$M_{WY}
=
\left(
\begin{array}{c|c}
 W_1(m_3;(e:m_1)^{\text{tuple}}) & 1\cdot W_1(m_3;(e-m_2:0)^{\text{tuple}})
\end{array}
\right)$. 
Combined with the inequality $m_1>e-m_2$ 
and \Cref{theorem:Wronskian}, 
it follows that 
$\det M_{WY}(\AAA,m)$ 
$=w((e:m_1)\oplus(e-m_2:0))\neq0$. 
Hence, we obtain $\exp(\AAA,m)=(\frac{|m|}{2},\frac{|m|}{2})$. 
\end{proof}

It may be useful to divide $M_{WY}$ into ``Power (monomial) part'' and 
``Wronski part''. 

\begin{definition}
For a tuple of nonnegative integers $a=(a_n,a_{n-1},\ldots,a_0)$ and $k\in\NN$, 
let us define the \emph{Power matrix} of $k$-rows and $(n+1)$-columns 
with respect to a variable $t$ by 
\[
P_t(k;a):=
\left(
\begin{array}{cccc}
 \varphi_{a_n}(t) & \varphi_{a_{n-1}}(t) & \cdots & \varphi_{a_{0}}(t) \\
 \varphi_{a_n-1}(t) & \varphi_{a_{n-1}-1}(t) & \cdots & \varphi_{a_{0}-1}(t) \\
 \vdots & \vdots & \ddots & \vdots \\
 \varphi_{a_n-(k-1)}(t) & \varphi_{a_{n-1}-(k-1)}(t) & \cdots & \varphi_{a_{0}-(k-1)}(t) 
 \end{array}
\right).
\]
\end{definition}

\begin{remark}
The $i$-th left block and right-block of $M_{WY}(\AAA,m)$ is expressed as 
\begin{itemize}
    \item $W_{z_i}(m_i;(e:m_1)) = P_{z_i}(m_i;(e:m_1)) \circ W(m_i;(e:m_1))$, 
    \item $z_i \cdot  W_{z_i}(m_i;(e-m_2:0)) = z_i\cdot P_{z_i}(m_i;(e-m_2:0)) \circ W(m_i;(e-m_2:0))$. 
\end{itemize}
\end{remark}

\begin{definition}\label{definition:WY-matrix_2}
Under the same assumption as \Cref{remark:WY-matrix_2}, 
let $P(\AAA,m)$ be the matrix, 
    \begin{itemize}
        \item whose $i$-th  left block is $P_{z_i}(m_i;(e:m_1))$ and 
        \item whose $i$-th right block is $z_i\cdot P_{z_i}(m_i;(e-m_2:0))$. 
    \end{itemize}
    Then, let $W(\AAA,m)$ be the matrix that satisfies 
    $M_{WY}(\AAA,m)=P(\AAA,m)\circ W(\AAA,m)$. 
\end{definition}

\begin{example}
    Let $(\AAA,m)$ be a multiarrangement defined by $Q=x^2y^2(x-y)^1(x-{z_4} y)^3$ 
    (see \Cref{example:M_WY:2213}). 
    Then, we have 
    $M_{WY}(\AAA,m)= {P}(\AAA,m) \hprod W(\AAA,m)$ as 
    \[
    \left(
    \begin{array}{cc:cc}
     1 & 1 & 1 & 1 \\
     \hdashline
     {z_4}^3 & {z_4}^2 & {z_4}^2 & {z_4}^1 \\ \cline{4-4}
     3{z_4}^2 & 2{z_4}^1 & {z_4}^1 & \multicolumn{1}{|c}{0} \\ \cline{3-3}
     6{z_4}^1 & 2 & \multicolumn{1}{|c}{0} & 0 
    \end{array}
    \right)
    =
    \left(
    \begin{array}{cc:cc}
     1 & 1 & 1 & 1 \\
    \hdashline
     {z_4}^3 & {z_4}^2 & {z_4}^2 & {z_4}^1 \\ \cline{4-4}
     {z_4}^2 & {z_4}^1 & {z_4}^1 & \multicolumn{1}{|c}{{z_4}^0} \\ \cline{3-3}
     {z_4}^1 & {z_4}^0 & \multicolumn{1}{|c}{{z_4}^0} & {z_4}^{-1}
    \end{array}
    \right)
    \hprod 
    \left(
    \begin{array}{cc:cc}
     1 & 1 & 1 & 1 \\
    \hdashline
     1 & 1 & 1 & 1 \\ \cline{4-4}
     3 & 2 & 1 & \multicolumn{1}{|c}{0} \\ \cline{3-3}
     6 & 2 & \multicolumn{1}{|c}{0} & 0 
    \end{array}
    \right).
    \]
    \end{example}

\begin{corollary}[Conjectured in \cite{Abe:B_2}, solved in \cite{2312.06356}]
\label{corollary: B_2 exponents in MN}
Let $(\AAA,m)$ be a balanced multiarrangement defined by 
$Q(\AAA,m)=x^{m_1}y^{m_2}(x-y)^{m_3}(x+y)^{m_4}$. 
If $|{m_2}-{m_1}| \geq |{m_4}-{m_3}|= 0$, then $\Delta(\AAA,m)=2$ 
if and only if ${m_1},{m_2}$ and $e:=\frac{|m|}{2}-1\in2\ZZ+1$. 
\end{corollary}

\begin{proof}
By \Cref{theorem:Saito's criterion,theorem:upper bound of delta}, 
it is enough 
to show that $\det M_{WY}(\AAA,m)=0$ if and only if 
${m_1},{m_2},e\in2\ZZ+1$, 
for the WY-matrix $M_{WY}$ where $|m|\in2\ZZ$ and $e:=\frac{|m|}{2}-1$. 
Note that $M_{WY}$ is a square matrix of size $m_3+m_4=(e-m_1+1)+(e-m_2+1)$ 
as mentioned in \Cref{remark:square case}.
We consider all possible cases for the parities 
of $m_1$, $m_2$ and $e$ as (i)-(iv) in \Cref{tabular:B_2}.
Recall that we divide $M_{WY}$ into two matrices $P(\AAA,m)$ and $W(\AAA,m)$ 
by \Cref{definition:WY-matrix_2}. 
In \Cref{tabular:B_2}, 
we use the symbols $P_{4,f}$ and $P_{4,g}$ as the lower blocks of $P(\AAA,m)$, 
\[
{P}(\AAA,m)=
\left(
\begin{array}{c|c}
 P_{1}(m_3;(e:m_1)) & P_{1}(m_3;(e-m_2:0)) \\
\hline
 P_{-1}(m_4;(e:m_1)) & -P_{-1}(m_4;(e-m_2:0)) 
\end{array}
\right)
=
\left(
\begin{array}{c|c}
 P_{3,f} & P_{3,g} \\
\hline
 P_{4,f} & P_{4,g} 
\end{array}
\right).
\]
\begin{table}[h]
\centering
\begin{tabular}{|c||c|c|c||c||c|c||c|}\hline
& $m_1$ & $m_2$ & $e:=\frac{|m|}{2}-1$ & $e-m_2$ 
& $P_{4,f}$ & $P_{4,g}$ & $|\det M_{WY}|/2^{m_3}$\\ \hline\hline
  (i) &  odd &  odd &  odd & even & $ J_1$ & $ J_1$ & 0 \\ \hline
 (ii) &  odd &  odd & even &  odd & $-J_2$ & $-J_2$ & $\star$ \\ \hline
(iii) & even & even &  odd &  odd & $ J_2$ & $-J_2$ & $\star$$\star$ \\ \hline
 (iv) & even & even & even & even & $-J_1$ & $ J_1$ & $\star$$\star$$\star$ \\ \hline
\end{tabular}
\vspace{1mm}
    \caption{The determinant $\det M_{WY}(\AAA,m)$ for each case}
  \label{tabular:B_2}
\end{table}
%
Similarly, 
let $M_{3,f},M_{3,g}$ and $M_{4,f},M_{4,g}$ 
denote each upper and lower blocks in $M_{WY}$.\\ 
Also, we define $J_1$ and $J_2$ by 
\[
J_1:=
\left(
\begin{array}{ccccc}
 -1 &  1 & \cdots &  1 & -1 \\
  1 & -1 & \cdots & -1 &  1 \\
 -1 &  1 & \cdots &  1 & -1 \\
 \vdots & \vdots &  & \vdots & \vdots 
\end{array}
\right),\ 
J_2:=
\left(
\begin{array}{ccccc}
 -1 &  1 & \cdots & -1 &  1 \\
  1 & -1 & \cdots &  1 & -1 \\
 -1 &  1 & \cdots & -1 &  1 \\
 \vdots & \vdots &  & \vdots & \vdots 
\end{array}
\right).
\]

Multiplying $(-1)^k$ for each $k$-th row of the lower block $\cevxx{M_{4,f}}{M_{4,g}}$, 
and then adding each $k$-th row of the upper block $\cevxx{M_{3,f}}{M_{3,g}}$ 
into the $k$-th row of the lower block, we can calculate the determinants as in \Cref{tabular:B_2}, 
where we use the symbols 
\begin{align*}
    \star &= 
     w((e-1:m_1)_{\text{odd}}\oplus(e-m_2-1:0)_{\text{even}})\cdot 
    w((e:m_1+1)_{\text{even}}\oplus(e-m_2:1)_{\text{odd}}), \\
    \star\star &= 
    w((e-1:m_1)_{\text{even}}\oplus(e-m_2:1)_{\text{odd}})\cdot 
    w((e:m_1+1)_{\text{odd}}\oplus(e-m_2-1:0)_{\text{even}}), \\
    \star\star\star &= 
    w((e-1:m_1+1)_{\text{odd}}\oplus(e-m_2:0)_{\text{even}})\cdot 
    w((e:m_1)_{\text{even}}\oplus(e-m_2-1:1)_{\text{odd}}).
    \end{align*}
By \Cref{theorem:Wronskian}, these three determinants are not zero. 
Hence, it follows that $\Delta(\AAA,m)=2$ if and only if $m$ satisfies the condition of Case (i). 
\end{proof}

\begin{remark}
    In \cite{Wakamiko,2312.06356}, 
    \Cref{corollary:A_2 exponents,corollary: B_2 exponents in MN} 
    were solved by constructing bases for the logarithmic derivation modules. 
    If we only focus on the exponents, the proofs above might suggest that 
    the matrix-method sometimes gives us an easier way 
    to calculate the exponents.    
\end{remark}
\section{Proof of the main results}\label{section:proof of main theorem}
To prove that $\Delta=0$, we consider the WY-matrix $M_{WY}$ with $e:=\frac{|m|}{2}-1$, which is square by \Cref{remark:square case}. 
Then, our goal is to show that $\det M_{WY} \neq 0$ 
by examining the prime factorization of each multiple of minors 
in \Cref{formula:Muir}, 
for all possible column assignments $\beta$ (also see \Cref{example:45313:proof}). 
Let us introduce some notation for the proof. 

\begin{definition}\label{definition:p-value}
    For $N\in\NN$ and a prime number $p$, define 
    $v_p(N):=\max\{n\in\ZZ_{\geq0} \mid \frac{N}{p^n}\in\ZZ\}$.
\end{definition}
\begin{example}
    $v_2(1)=v_2(13)=0$, 
    $v_2(8)=v_2(24)=3$. 
    For any prime number $p\in\NN$ and $N_1,N_2\in\ZZ$, we have 
    $v_p(N_1N_2)=v_p(N_1)+v_p(N_2)$ and $v_p(N_1+N_2)\geq\min\{v_p(N_1),v_p(N_2)\}$.
\end{example}

When we utilize \Cref{formula:Muir} for the proofs, 
the following assignment function $\beta_E$ plays an important role. 

\begin{definition}\label{definition:beta_E}
Let $M$ be a square matrix of size $m$, and let 
$\gamma=(\gamma_1,\gamma_2,\ldots,\gamma_k)$ be an ordered partition. 
Then, define the \emph{basic assignment} 
$\binom{m}{\gamma}\ni\beta_E: \{1,2,\ldots,m\} \to \{1,2,\ldots,k\}$ 
as 
$\beta_E(t)=j$ for $(\sum_{i<j}\gamma_{i}+1\le t\le \sum_{i\leq j}\gamma_{i})$. 
\end{definition}

Now, we are ready to prove \Cref{theorem:main theorem}.
%
\begin{proof}[Proof of \Cref{theorem:main theorem}]
We consider the case where the multiarrangement $(\AAA,m)$ satisfies 
$|s_n|\neq1$, $\gcd(s_i,s_n)=1$ for $3\leq i\leq n-1$ and $m_2+m_n=\frac{|m|}{2}$. 
Consider the WY-matrix $M_{WY}(\AAA,m)$, which is square of size $\sum_{i=3}^{n}m_i$ 
with $e:=\frac{|m|}{2}-1$. 
Let $\delta_{xy}:=\delta_{\ker x}+\delta_{\ker y}$. 
Then, it suffices to show that 
$\det M_{WY}(\AAA,m+k\cdot\delta_{xy})\neq0$ 
for sufficiently large $k\in\NN$. 

For simplicity, let us label each block of $M_{WY}$ as 
\[
M_{WY}(\AAA,m)
=
\left(
\begin{array}{c|c}
 A(m) & B(m) \\
 \hline
 C(m) & D(m)
\end{array}
\right)
=
\left(
\begin{array}{c:c:c:c|c}
 M_{3,3} & M_{3,4} & \cdots & M_{3,n-1} & M_{3,n} \\
 \hdashline
 M_{4,3} & M_{4,4} & \cdots & M_{4,n-1} & M_{4,n} \\
 \hdashline
 \vdots  & \vdots & \ddots & \vdots & \vdots \\
 \hdashline
 M_{n-1,3} & M_{n-1,4} & \cdots & M_{n-1,n-1} & M_{n-1,n} \\
 \hline
 M_{n,3} & M_{n,4} & \cdots & M_{n,n-1} & M_{n,n} 
\end{array}
\right), 
\]
where each diagonal block $M_{i,i}$ is a square matrix of size $m_i$ 
and $D(m)=M_{n,n}$. 
Since 
$\sum_{i=3}^{n-1}m_i=\frac{|m|}{2}-m_1
=e-m_1+1$ 
and 
$m_n=\frac{|m|}{2}-m_2
=e-m_2+1$, 
it follows that 
$\{M_{i,j} \mid 3\leq j \leq n-1\}$ and 
$\{M_{i,n}\}$ are entirely 
in the $f$-part (left) and the $g$-part (right), respectively. 
By \Cref{remark:WY-matrix_2}, 
replacing $m$ to $m':=m+\delta_{xy}$ 
does not change the right blocks $M_{i,n}$, or 
$B(m')=B(m)$ and $D(m')=D(m)$. 
On the other hand, 
for the left blocks, we have 
$\det A(m')= (\prod_{i=3}^{n-1}s_i^{m_i})\cdot\det A(m)$ and 
$\min\{v_{s_n}(a) \mid a\in C(m')\}=1+\min\{v_{s_n}(a) \mid a\in C(m)\}$ 
(we use \Cref{formula:Muir} for $\det A(m)$ with the ordered partition $\gamma'=(m_3,m_4,\ldots,m_{n-1})$). 
Hence, when we consider the case $\det A(m)\neq0$, 
$v_{s_n}(\det A(m+k\cdot\delta_{xy}) \cdot \det D(m+k\cdot\delta_{xy}))
< v_{s_n}(\det(\gamma,\beta;1)\cdot\det(\gamma,\beta;2) )$
where $\beta\in\binom{\sum_{i=3}^{n}m_i}{\gamma}\setminus\{\beta_E\}$
with $\gamma=(\sum_{i=3}^{n-1}m_i, m_n)$, 
which derives 
$\det M_{WY}(\AAA, m+k\cdot\delta_{x y})\neq0$ for sufficiently large $k$. 
\end{proof}

\begin{example}
\label{example:45313:proof}
Let $(\AAA,m(k))$ be a multiarrangement defined by 
$Q(\AAA,m(k))=x^{k}y^{k+1}(x-y)^3(x-3y)^1(x-2y)^3$, $|m(k)|=2k+8$. 
Let $e=k+3$ and consider $M=M_{WY}(\AAA,m(k))$. 
%
\[
M
=
\left(
\begin{array}{ccc:c|ccc}
  1 &  1 &  1 &  1 &  1 &  1 &  1 \\ 
  k+3 &  k+2 &  k+1 &  k &  2 &  1 &  0 \\ 
 (k+3)_2 & (k+2)_2 & (k+1)_2 & (k)_2 & 2 & 0 &  0 \\ 
 \hdashline
 3^{k+3} & 3^{k+2} & 3^{k+1} & 3^{k} & 3^3 & 3^2 & 3^1 \\ 
  \hline
 2^{k+3} & 2^{k+2} & 2^{k+1} & 2^{k} & 2^3 & 2^2 & 2^1 \\ 
 (k+3)\cdot2^{k+2} & (k+2)\cdot2^{k+1} & (k+1)\cdot2^{k} & (k)\cdot2^{k-1} & 2\cdot2^2 & 2^1 & 0 \\ 
 (k+3)_2\cdot2^{k+1} & (k+2)_2\cdot2^{k} & (k+1)_2\cdot2^{k-1} & (k)_2\cdot2^{k-2} & 2\cdot2^1 & 0 & 0 
\end{array}
\right). 
\]
Note that $(x)_n$ denote the factorial $(x)_n=\prod_{i=0}^{n-1} (x-i)$, 
and we use the block symbols $A(m(k)),B(m(k)),C(m(k)),D(m(k))$ 
as in the proof of \Cref{theorem:main theorem}. 
%
By \Cref{remark:WY-matrix_2}, 
the right blocks do not depend on $k$. 
%
%
%
Since 
$\det A(m(k))\cdot D(m(k))=2^8\cdot3^{k}\neq0$ and 
$\min\{v_{2}(a) \mid a\in C(m(k))\}\geq {k-1}$, 
by \Cref{formula:Muir} for $\gamma=(4,3)$, 
$k\geq10$ implies $\det M\neq0$. 
\end{example}

Next, we prove the rest 
results, concerning the Coxeter multiarrangement of type $B_2$. 
\Cref{proposition:minimum_value} plays one of the most important roles in the proof, 
which can be shown by 
repeatedly applying the pigeonhole principle. 
\begin{proposition}
\label{proposition:minimum_value}
For any descending nonnegative tuple $a=(a_n,a_{n-1},\ldots,a_0)$ 
of length $n+1$ 
and for any prime number $p\in\NN$, we have 
$v_p(w(a))\geq v_p(w(\tuple{n}))$. 
\end{proposition}

Note that the multiarrangement defined by $x^{a}y^{b}(x-y)^c(x+y)^d$ is equivalent to the multiarrangement $X^dY^b(X-Y)^a(X-2Y)^c$ as mentioned in \Cref{remark:transformation}. 
Hence, we consider the multiarrangement $(\BBB,m)$ with the defining polynomial $Q(\BBB,m)=x^dy^b(x-y)^a(x-sy)^c$ with $s\in\ZZ\setminus\{0,\pm1\}$. 
Now we are ready for the proofs. 

\begin{proof}[Proof of \Cref{corollary:main theorem,theorem:B2exponents_new}]
    Let $(\BBB,m)$ be the multiarrangement $Q(\BBB,m)=x^dy^b(x-y)^a(x-sy)^c$ 
    that satisfies $s\in\ZZ\setminus\{0,\pm1\}$, 
    $b-a\geq d-c\geq1$ and $b-a=(d-c)+2k$ with $k\in\ZZ_{\geq0}$. 
    Note that this equation implies $|m|\in2\ZZ$. 
    Then, it suffices to show that 
    $\det M_{WY}(\BBB,m)\neq0$ with $e=\frac{|m|}{2}-1$. 
Let us label each block of the WY-matrix as 
\[
M_{WY}(\BBB,m)
=
\left(
\begin{array}{c|c}
 M_{3,3} & M_{3,4} \\
 \hline
 M_{4,3} & M_{4,4}
\end{array}
\right)
=
\left(
\begin{array}{cc|c}
 \boxed{A} & B & C \\
\hline
 D & E & \boxed{F} \\
\hdashline
 G & \boxed{H} & I 
\end{array}
\right)
\]
where 
%
$M_{3,3}=(A \ \ B)$ and 
$A$ is a square matrix of size $a$, 
$F$ is a square matrix of size $c-k$ and 
$H$ is a square matrix of size $k$. 
Since $b-a=d-c+2k$ and $a+b+c+d=|m|=2e+2$, we have $c-(e-b+1)=k$. 
Similarly to $M_{WY}(\BBB,m)$, for $P(\BBB,m)$ and $W(\BBB,m)$, 
we define blocks $P_{i,j}$ and $W_{i,j}$ to be 
$M_{i,j}=P_{i,j}\hprod W_{i,j}$. 
The boxed blocks $A,F,H$ correspond to the submatrices 
$(\gamma,\beta;1), (\gamma,\beta;2), (\gamma,\beta;3)$, respectively, 
under the row ordered partition $\gamma=(a,c-k,k)$ and the assigning function $\beta$ defind later. 
Note that $\min\{v_s(a) \mid a\in P_{4,3}\}=d-c+1\geq2$. \\
\[\small
P(\BBB,m)
=
\left(
\begin{array}{cccc|cccc}
\cline{1-2}
 \multicolumn{1}{|c}{1} & \multicolumn{1}{c|}{\cdots} & \cdots & 1 & 1 & \cdots & 1 & 1 \\ 
  \multicolumn{1}{|c}{1} & \multicolumn{1}{c|}{\cdots} & \cdots & 1 & 1 & \cdots & 1 & 1 \\
 \multicolumn{1}{|c}{\vdots} & \multicolumn{1}{c|}{\ddots} & \ddots & \vdots & \vdots & \ddots &\vdots & \vdots \\
 \multicolumn{1}{|c}{1} & \multicolumn{1}{c|}{\cdots} & \cdots & 1 & 1 & \cdots & 1 & 1 \\ 
 \cline{1-2}
\hline
\cline{5-8}
 s^e & \cdots & \cdots & s^d & \multicolumn{1}{|c}{s^{c-k}} & \cdots & s^2 & \multicolumn{1}{c|}{s^1} \\
 s^{e-1} & \cdots & \cdots & s^{d-1} & \multicolumn{1}{|c}{s^{c-k-1}} & \cdots & s^1 & \multicolumn{1}{c|}{\gray{s^0}} \\
 \vdots & \ddots & \ddots & \vdots & \multicolumn{1}{|c}{\vdots} & \ddots & \gray{\vdots} & \multicolumn{1}{c|}{\gray{\vdots}} \\
 s^{e-c+k+1} & \cdots & \cdots & s^{d-c+k+1} & \multicolumn{1}{|c}{s^1} & \gray{\cdots} & \gray{s^{-(c-k-3)}} & \multicolumn{1}{c|}{\gray{s^{-(c-k-2)}}} \\
 \cline{5-8}
 \hdashline
 \cline{3-4}
 s^{e-c+k} & \cdots & \multicolumn{1}{|c}{\cdots} & \multicolumn{1}{c|}{s^{d-c+k}} & \gray{0} & \cdots & \gray{0} & \gray{0} \\
\vdots & \ddots & \multicolumn{1}{|c}{\ddots} & \multicolumn{1}{c|}{\vdots} & \vdots & \ddots &\vdots & \vdots \\
 s^{e-c+1} & \cdots & \multicolumn{1}{|c}{\cdots} & \multicolumn{1}{c|}{s^{d-c+1}} & \gray{0} & \gray{\cdots} & \gray{0} & \gray{0}\\
 \cline{3-4}
\end{array}
\right),
\]
\[\small
W(\BBB,m)
=
\left(
\begin{array}{cccc|cccc}
\cline{1-2}
 \multicolumn{1}{|c}{1} & \multicolumn{1}{c|}{\cdots} & \cdots & 1 & 1 & \cdots & 1 & 1 \\ 
  \multicolumn{1}{|c}{e} & \multicolumn{1}{c|}{\cdots} & \cdots & d & e-b & \cdots & 1 & 0 \\
 \multicolumn{1}{|c}{\vdots} & \multicolumn{1}{c|}{\ddots} & \ddots & \vdots & \vdots & \ddots &\vdots & \vdots \\
 \multicolumn{1}{|c}{\frac{e!}{(e-a+1)!}} & \multicolumn{1}{c|}{\cdots} & \cdots & \frac{d!}{(d-a+1)!} & \frac{(e-b)!}{(e-b-a+1)!} & \cdots & 0 & 0 \\ 
 \cline{1-2}
\hline
\cline{5-8}
 1 & \cdots & \cdots & 1 & \multicolumn{1}{|c}{1} & \cdots & 1 & \multicolumn{1}{c|}{1} \\
 e & \cdots & \cdots & d & \multicolumn{1}{|c}{e-b} & \cdots & 1 & \multicolumn{1}{c|}{\gray{0}} \\
 \vdots & \ddots & \ddots & \vdots & \multicolumn{1}{|c}{\vdots} & \ddots & \gray{\vdots} & \multicolumn{1}{c|}{\gray{\vdots}} \\
 \frac{e!}{(e-c+k+1)!} & \cdots & \cdots & \frac{d!}{(d-c+k+1)!} & \multicolumn{1}{|c}{(c-k-1)!} & \gray{\cdots} & \gray{0} & \multicolumn{1}{c|}{\gray{0}} \\
 \cline{5-8}
 \hdashline
 \cline{3-4}
 \frac{e!}{(e-c+k)!} & \cdots & \multicolumn{1}{|c}{\cdots} & \multicolumn{1}{c|}{\frac{d!}{(d-c+k)!}} & \gray{0} & \cdots & \gray{0} & \gray{0} \\
\vdots & \ddots & \multicolumn{1}{|c}{\ddots} & \multicolumn{1}{c|}{\vdots} & \vdots & \ddots &\vdots & \vdots \\
 \frac{e!}{(e-c+1)!} & \cdots & \multicolumn{1}{|c}{\cdots} & \multicolumn{1}{c|}{\frac{d!}{(d-c+1)!}} & \gray{0} & \gray{\cdots} & \gray{0} & \gray{0}\\
 \cline{3-4}
\end{array}
\right).
\]
\begin{enumerate}
    \item [(i)] When $k=0$, 
note that the blocks $G,H,I,B,E$ vanish 
and $M_{i,i}$ become square. 
In this case, we consider the row ordered partition $\gamma=(a,c)$. 
By \Cref{proposition:minimum_value}, 
$v_s(\det M_{3,3} \cdot \det M_{4,4})<v_s(\det(\gamma,\beta;1)\cdot \det(\gamma,\beta;2))$ where $\beta \in\binom{a+c}{\gamma}\setminus\{\beta_{E}\}$. 
Hence, $\det(M_{WY}(\BBB,m))\neq0$ by \Cref{formula:Muir}, which derives $\Delta(\BBB,m)=0$. Substituting $s=2$, we obtain the result. 

    \item [(ii)] When $k\geq1$, 
we consider the row ordered partition $\gamma=(a,e-b+1,c-(e-b+1))=(a,c-k,k)$. 
Let $\beta:\{1,2,\ldots,a+c\}\to\{1,2,3\}$ be the assigning function defined by
%
$\beta^{-1}(\{1\})=\{1,2,\ldots,a\}$, 
$\beta^{-1}(\{2\})=\{a+k+1,a+k+2,\ldots,a+c\}$ and 
$\beta^{-1}(\{3\})=\{a+1,a+2,\ldots,a+k\}$. 
For these $\gamma$ and $\beta$, we have 
$\det(\gamma,\beta;3)=\det H
=s^{k(d-c)+k^2} \cdot \left(\prod_{i=1}^{k}(c-i)!\right) \cdot  \det K$ 
where 
$K=(\binom{d+k-j}{c-k-1+i})_{1\leq i,j\leq k}$, 
and for other $\beta'\in \binom{a+c}{\gamma}$, 
we have 
$\det(\gamma,\beta';3)=0$ or 
$v_s(\det(\gamma,\beta';3))> ((k\cdot(d-c)+k^2))+\sum_{i=1}^{k} v_s(c-i)!$. 
Then, if $v_s(\det K)=0$, we have 
$v_s(\det(\gamma,\beta;1)\cdot\det(\gamma,\beta;2)\cdot\det(\gamma,\beta;3))
< v_s(\det(\gamma,\beta';1)\cdot\det(\gamma,\beta';2)\cdot\det(\gamma,\beta';3))$
for other $\beta'\in \binom{a+c}{\gamma}$, 
so $v_s(\det M_{WY})\neq0$. 
%
Substituting $s=2$ 
and since 
$K=(\binom{\frac{|m|}{2}-m_1-j}{\frac{|m|}{2}-m_2-1+i})_{1\leq i,j\leq k}$, 
we obtain the result. 
\vspace{-8.5mm}
\end{enumerate}
\end{proof}

\Cref{theorem:B2exponents_new} is equivalent to the following. 
\begin{corollary}\label{corollary:thm1.8}
    Let $(\AAA,m)$ be the multiarrangement defined by 
    $Q(\AAA,m)=x^{m_1}y^{m_2}(x-y)^{m_3}(x+y)^{m_4}$ 
    where  
    $m_2-m_1=(m_4-m_3)+2k$, $m_4-m_3\geq0$ and $k\in\ZZ_{\geq0}$. 
    If 
    equality holds in the following inequality, then we have $\Delta=0$: 
    $$v_2(\prod_{j=1}^{k-1}(m_3-j)^j)
    \leq v_2(\prod_{j=0}^{k-1}\binom{m_4+j}{m_3-k}\cdot \prod_{0\leq i<j\leq k-1}(j-i)).$$
\end{corollary}

\begin{example}\label{example:thm1.8}
    When $k=2$, 
    the inequality is expressed as 
    $v_2(m_3-1)\leq v_2(\binom{m_4}{m_3-2}\cdot\binom{m_4+1}{m_3-2})$. 
    Using a computer, we can verify which multiplicities
    hold the equality. 
    For example, 
    the list is as follows 
    for $7\leq m_3\leq 9$ and $m_3+1 \leq m_4 \leq 100$. 
%
\begin{align*} 
(m_3,m_4)\in
\{
&(7,13),(7,14),(7,21),(7,22),(7,29),(7,30),(7,37),(7,38),(7,45),(7,46),(7,53),\\
&(7,54),(7,61),(7,62),(7,69),(7,70),(7,77),(7,78),(7,85),(7,86),(7,93),(7,94),\\
&(8,14),(8,22),(8,30),(8,38),(8,46),(8,54),(8,62),(8,70),(8,78),(8,86),(8,94),\\
&(9,14),(9,23),(9,30),(9,39),(9,46),(9,55),(9,62),(9,71),(9,78),(9,87),(9,94)
\}.
\end{align*}
%
%
Hence, for example, $m=(k,k+10,7,13)$ satisfies $\Delta(m)=0$ 
for any $k\geq1$. 
\end{example}




\section*{Acknowledgements}
The author would like to thank Takuro Abe 
for his many helpful comments and valuable advice. 
The author was supported by WISE program (MEXT) at Kyushu University. 
\vspace{-2mm}
\bibliographystyle{amsplain}
\bibliography{bib_mr,bib_arxiv}

\end{document}